\documentclass{amsart}

\usepackage{amsfonts}
\usepackage{hyperref}

\newtheorem{thm}{Theorem}[section]

\newtheorem{lema}[thm]{Lemma}
\newtheorem{prop}[thm]{Proposition}
\theoremstyle{definition}
\newtheorem{defn}[thm]{Definition}
\theoremstyle{remark}
\newtheorem{rem}[thm]{Remark}

\numberwithin{equation}{section}

\newcommand{\R}{\mathbb R}
\newcommand{\N}{\mathbb N}

\newcommand{\s}{\mathbf{s}}
\newcommand{\p}{\mathbf{p}}
\newcommand{\q}{\mathbf{q}}
\newcommand{\rr}{\mathbf{r}}
\newcommand{\Q}{\mathcal{Q}}
\newcommand{\F}{\mathcal{F}}
\newcommand{\LL}{\mathcal{L}}

\newcommand{\ve}{\varepsilon}
\def\supp{\mathop{\text{\normalfont supp}}}
\def\diver{\mathop{\text{\normalfont div}}}
\begin{document}
\title[Anisotropic eigenvalues]{Existence of Eigenvalues for Anisotropic and Fractional Anisotropic Problems via Ljusternik-Schnirelmann Theory}
	
\author[I. Ceresa Dussel and J. Fern\'andez Bonder]{Ignacio Ceresa Dussel and Juli\'an Fern\'andez Bonder}
	
\address{Instituto de C\'alculo, CONICET\\
Departamento de Matem\'atica, FCEN - Universidad de Buenos Aires\\
Ciudad Universitaria, 0+$\infty$ building, C1428EGA, Av. Cantilo s/n\\
Buenos Aires, Argentina}
	
\begin{abstract}
In this work, our interest lies in proving the existence of critical values of the following Rayleigh-type quotients
$$
\Q_{\p}(u) = \frac{\|\nabla u\|_{\p}}{\|u\|_{\p}},\quad\text{and}\quad \Q_{\s,\p}(u) = \frac{[u]_{\s,\p}}{\|u\|_{\p}},
$$
where $\p = (p_1,\dots,p_n)$, $\s=(s_1,\dots,s_n)$ and 
$$
\|\nabla u\|_{\p} = \sum_{i=1}^n \|u_{x_i}\|_{p_i}
$$
is an anisotropic Sobolev norm, $[u]_{\s,\p}$ is a fractional version of the same anisotropic norm, and
$$
\|u\|_{\p} =\left(\int_{\R}\left(\dots \left(\int_{\R}|u|^{p_1}dx_1\right)^{\frac{p_2}{p_1}}\,dx_2\dots \right)^{p_n/p_{n-1}}dx_n\right)^{1/p_n}
$$
is an anisotropic Lebesgue norm.
	
Using the Ljusternik-Schnirelmann theory, we prove the existence of a sequence of critical values and we also find an associated Euler-Lagrange equation for critical points.  Additionally, we analyze the connection between the fractional critical values and  its local counterparts.
\end{abstract}
	
	
	

\maketitle

\section{Introduction}

Eigenvalue problems are a well-established and widely studied subject that spans across various fields, including analysis and partial differential equations (PDEs). In the context of PDEs, the Laplacian eigenvalue problem involves finding the eigenvalues and corresponding eigenfunctions that satisfy the equation
$$
\Delta u+\lambda u=0.
$$

Solving this problem provides valuable insights into the behavior of functions within a given domain. The eigenvalues offer information about the Laplacian's spectrum, while the eigenfunctions reveal spatial patterns associated with different frequencies or modes. 



Another well-known eigenvalues problem arises with the $p$-Laplacian, a nonlinear generalization of the Laplacian defined by
$$
\Delta_p u=\diver(|\nabla u|^{p-2}\nabla u).
$$
The eigenvalue problem of the $p$-Laplacian, characterized by the nonlinearity introduced through power exponentiation, is both intriguing and demanding, as it entails solving the equation
$$
-\Delta_p u=\lambda |u|^{p-2}u
$$  
Many authors have extensively explored this problem, as seen in \cite{BONDERROSSI02,LEAN,LINDQVIST13,BELLONI04}.

Even more, in recent decades, there has been a growing interest in fractional operators due to their applications in various natural sciences models \cite{CORDOBA04,GONZALVES19,RIASCO15}. One prominent exponent of this family of fractional operators is the fractional $p$-Laplacian, defined as
$$
(-\Delta_p)^s u(x) = \text{p.v.} (1-s)K_{n,p}\int_{\R^n} \frac{|u(x)-u(y)|^{p-2}(u(x)-u(y))}{|x-y|^{n+sp}}\, dy,
$$ 
where $K_{n,p}$ is a constant that depends only on $n$ and $p$.

The eigenvalue problem associated with this fractional operator,
$$
(-\Delta_p)^s u(x)=\lambda |u|^{p-2}u,
$$
has also been studied extensively by several authors \cite{CAFA07,DINEZZA12,VALDINOCI13,FRANZINA14}. 
 
 The aim of this paper is to introduce an anisotropic feature to these eigenvalue problems. This choice is motivated by the substantial attention dedicated to investigating this phenomenon in signal processing and diffusion studies \cite{CHERNYSHOV2018,TSIOTSIOS2013}. For those not familiar with the term, anisotropy can be described as the characteristic of displaying directional dependence, where various attributes or qualities manifest differently in distinct directions. This stands in opposition to the isotropic nature of the Laplacian, $p-$Laplacian, and fractional $p-$Laplacian, where these properties remain uniform regardless of the direction.

On one hand, a strategy to address anisotropy involves emphasizing the integrability of individual partial derivatives of a function $u$ by employing the sum of standard $L^p$ norms,
$$
\|\nabla u\|_{\p}=\sum_{i=1}^{n}\left\|u_{x_i}\right\|_{p_i},
$$
 see \cite{RAKOSNIK79,RAKOSNIK81,TROISI69}. Hence, we naturally arrive at the following anisotropic pseudo-laplace operator
$$
-\widetilde{\Delta}_\p u := -\diver\left(\sum_{i=1}^{n}|u_{x_i}|^{p_i-2}u_{x_i}\right)
$$
On the other hand, Benedek $\&$ Panzone  \cite{BENEDEK61} present the anisotropic $L^{\p}$ ($\p=(p_1,\dots,p_n)$) space with a special norm  to address the anisotropy of a function $u$. The mixed Lebesgue space is constructed by considering different exponents for each coordinate in the norm
$$
\|u\|_{\p}=\left(\int_{\R}\left(\dots \left(\int_{\R}|u|^{p_1}dx_1\right)^{\frac{p_2}{p_1}}\,dx_2\dots \right)^{p_n/p_{n-1}}dx_n\right)^{1/p_n}.
$$
By considering different exponents for each coordinate, the mixed Lebesgue norm accounts for the anisotropy of the function $u$. It allows for a more flexible and nuanced characterization of the integrability and decay properties across different coordinates.

By combining this two perspective we can state the following eigenvalue problem 
$$
-\widetilde{\Delta}_\p u =\lambda \F_{\p}(u),
$$
where $\F_{\p}$ is a suitable functional related to $\|u\|_{\p}$. See \eqref{F_p}.

Unfortunately, this problem is hindered by its lack of homogeneity. It's important to observe that if $v$ is an eigenfunction associated with $\lambda$, there's a possibility that $tv$ may not qualify as an eigenfunction of $\lambda$.

Note that a crucial approach to solving the Laplacian, $p-$Laplacian and fractional $p-$Laplacian eigenvalue problems involves finding the critical points of the Rayleigh quotient associated with each one, namely
$$\Q_2(u) = \frac{\|\nabla u\|^2_{2}}{\|u\|^2_{2}},\quad\Q_p(u) = \frac{\|\nabla u\|_p^p}{\|u\|_p^p}\quad\text{and }\Q_{sp}(u) = \frac{[u]_{sp}^p}{\|u\|_p^p}. $$  

Therefore, it is recommended to explore the following homogeneous Rayleigh quotient
\begin{equation}\label{Qp}
\Q_{\p}(u)=\frac{\|\nabla u\|_{\p}}{\|u\|_{\p}}.
\end{equation}
As we will observe in Section \ref{section3}, the associated Euler-Lagrange equation of $\Q_{\p}(u)$ is the following homogeneous eigenvalue problem,
\begin{equation}\label{eigenvalue}
-\LL_{\p}u=-\diver\left(\sum_{i=1}^{n}\left|\frac{u_{x_i}}{\|u_{x_i}\|_{p_i}}\right|^{p_i-2}\frac{u_{x_1}}{\|u\|_{p_i}}\right)=\lambda \F_{\p}(u).
\end{equation}

\bigskip

In \cite{CHAKER23} fractional anisotropy is introduced through the utilization of integrability parameters $\p=(p_1,\ldots,p_n)$, $1<p_i<\infty$, and fractional parameters $\s=(s_1,\ldots,s_n)$, $0<s_i<1$, and the subsequent norm,

$$
[u]_{\s,\p}=\sum_{i=1}^{n}\left( \int_{\R^n}\int_{\R} (1-s_i) \frac{|u(x+he_i)-u(x)|^{p_i}}{|h|^{1+s_ip_i}}\,dh \,dx\right)^{1/p_i}.
$$

As in the non-fractional case, combining this perspective with the  Benedek $\&$ Panzone's norm we arrive to the following eigenvalue problem
$$
(-\widetilde{\Delta}_{\p})^{\s} u(x)=\lambda \F_{\p}(u),
$$
where $(-\widetilde{\Delta}_{\p})^{\s}$ it the \textit{fractional pseudo p-Laplacian} operator defined as
$$
(-\widetilde{\Delta}_{\p})^{\s} u(x) =\sum_{i=1}^{n} \int_{\R^n}\frac{(1-s_i)}{p_i} \frac{|u(x+he_i)-u(x)|^{p_i-2}(u(x+he_i)-u(x))}{|h|^{1+sp}}\,dh\,dx .
$$
Again, this is not an homogeneous problem, therefore we study the
 homogeneous Rayleight quotient 
\begin{equation}\label{Qsp}
	\Q_{\s,\p}(u)=\frac{[u]_{\s,\p}}{\|u\|_{\p}}.
\end{equation}
As we will see in Section \ref{section3} the Euler-Lagrange equation is
\begin{equation}\label{eigenvalue_s}
	-\LL_{\s. \p}u=\lambda \F_{\p}(u),
\end{equation}
where $\LL_{\s,\p}$ is the fractional version of $\LL_{\p}$.

To address the problem of find criticals points of \eqref{Qp}, and \eqref{Qsp} and solve the eigenvalues problems \eqref{eigenvalue} and \eqref{eigenvalue_s}, the Ljusternik-Schnirelman theory serves as a powerful framework for exploring critical point theory and the existence of critical points for functionals as we will see in Section \ref{section5}. See \cite{MOTREANU14}.

The rest of the paper is organized as follows: In Section \ref{section2}, we dive into anisotropic Sobolev spaces and fractional anisotropic Sobolev spaces, explaining them in more detail and discussing some interesting properties like the Poincaré inequality and a Rellich-Kondrashov type theorem. Then, in Section \ref{section3}, we figure out the Euler-Lagrange equations associated with the corresponding Rayleight-type quotients. In Section \ref{section4} we study the asymptotic behavior of the sequence of eigenvalues as $\s\to(1,\dots,1)$ and finally in Section \ref{section5} we use Ljusternik-Schnirelman theory to prove the existence of eigenvalues.

\section{Mixed, anisotropic and fractional  spaces}\label{section2}
In this section, our objective is to establish the definition of the mixed Lebesgue space, as introduced by \cite{BENEDEK61}. This space will serve as a fundamental building block for our analysis.
Furthermore, we will define a suitable anisotropic Sobolev space, $W_0^{1,\p}(\Omega)$, and a fractional anistropic Sobolev space $W_0^{\s,\p}(\Omega)$.

\subsection{Mixed space}

Let $\p=(p_1,p_2,\dots,p_n)$ with $1<p_i<\infty$ for $i=1,\dots,n$ be integral parameters. Without loss of generality, we can assume that 
\begin{equation}\label{condicion}
1<p_1\leq p_2 \leq \dots\leq p_n<\infty.
\end{equation}  
We define the \textit{mixed Lebesgue space} as
 $$
 L^{\p}(\R^n)=\{u\text{ measurable such that } \|u\|_{L^{\p}(\R^n)}<\infty\}.$$ Where 
$$
\|u\|_{\p}=\left(\int_{\R}\left(\dots \left(\int_{\R}|u|^{p_1}dx_1\right)^{\frac{p_2}{p_1}}\,dx_2\dots \right)^{p_n/p_{n-1}}dx_n\right)^{1/p_n}.
$$
Furthermore, given $\Omega$ an open bounded subset of $\R^n$, we define 
$$
 L^{\p}(\Omega)=\{u\in L^{\p}(\R^n)\text{ such that }u=0 \text{ in } \R^n\setminus\Omega\}. 
$$
Observe that $L^{\p}(\Omega)$ is a closed subspace of $L^{\p}(\R^n)$. This space $L^{\p}(\Omega)$ turns out to be a reflexive Banach space and its properties were studied in \cite{BENEDEK61,ADAMS88}.
\begin{rem}\label{rem.recurrencia}
The $\|.\|_{\p}$ norm can be defined by recurrence as
\begin{align*}
I_1(u)&=\left(\int_{\R}|u|^{p_1}\,dx_1\right)^{1/p_1}\\
I_2(u)&=\left(\int_{\R}I_1(u)^{p_2}\,dx_2\right)^{1/p_2}\\
&\vdots\\
I_j(u)&=\left(\int_{\R}I_{j-1}(u)^{p_j}\,dx_j\right)^{1/p_j}\\
&\vdots\\
I_n(u)&=\left(\int_{\R}I_{n-1}(u)^{p_n}\,dx_n\right)^{1/p_n}\\
I(u)&= I_n(u) =\|u\|_{\p}
\end{align*}

\end{rem}

\begin{rem}\label{rem.norma parcial}
Observe that, given  $u\in L^{\p}(\R^n)$, $I_j(u)$ is a function of $(x_{j+1},\dots,x_n)$. 

Moreover, for almost every $(y_{j+2},\dots,y_n)\in \R^{n-j-2}$, the function $I_j(u)$ (as a function of $x_{j+1}$), belongs to $L^{p_{j+1}}(\R)$.
	
Also, observe that if $\{u_k\}_{k\in \N}\subset L^{\p}(\R^n)$ is such that $u_k\to u \in L^{\p}(\R^n)$ as $k\to\infty$ then $I_j(u_k)(\cdot, y_{j+2},\dots,y_n)\to I_j(u)(\cdot, y_{j+2},\dots,y_n)$ in $L^{p_{j+1}}(\R)$ for a.e. $(y_{j+2},\dots,y_n)\in \R^{n-j-2}$.
\end{rem} 

\subsection{Anisotropic Sobolev spaces}

Our interest lies in functions whose partial derivatives have different integrability. With this fact in mind, given $\p = (p_1,\dots,p_n)$ with $1<p_i<\infty$, the anisotropic Sobolev space is defined as follows:
$$
W^{1,\p}(\R^n):=\left\{u\in L^{\p}(\R^n) \colon  u_{x_i} \in L^{p_i}(\R^n),\ i=1,\dots,n\right\},
$$
 equipped with the following norm
$$
\| u\|_{1,\p} = \|u\|_{\p} + \sum_{i=1}^{n}\left\|u_{x_i}\right\|_{p_i} = \|u\|_{\p} + \|\nabla u\|_{\p}.
$$
It is easy to prove that $W^{1,\p}(\R^n)$ is a separable, reflexive Banach space. 

Now, given a bounded domain $\Omega\subset \R^n$ we define $W^{1,\p}_0(\Omega)$ as the closure of $C_c^\infty(\Omega)$ in $W^{1,\p}(\R^n)$.

\subsection{Fractional space}
Next we present the fractional anisotropic Sobolev space.

First, given $i=1,\dots, n$,  $s\in(0,1]$ and $p\in (1,\infty)$, for any $u\colon \R^n\to\R$ measurable we define the quantity
$$
[u]_{s,p,i} = 
\left(\int_{\R^n}\int_\R\frac{|u(x+he_i)-u(x)|^p}{|h|^{1+sp}}\, dh dx\right)^{\frac{1}{p}},
$$
where $e_i$ is the $i^\text{th}-$canonical vector base in $\R^n$.

Now, given $\p=(p_1,\dots,p_n)$ and $\s=(s_1,\dots,s_n)$ with $1<p_i<\infty$ and $0<s_i<1$, for $i=1,\dots,n$, we define the anisotropic fractional order Sobolev space as
$$
W^{\s,\p}(\R^n):=\left\{u\in L^{\p}(\R^n)\colon [u]_{s_i,p_i,i}<\infty,\ i=1,\dots,n\right\}.
$$
This space has a natural norm defined as
$$
\|u\|_{\s,\p} := \|u\|_{\p} + \sum_{i=1}^n [u]_{s_i,p_i,i} = \|u\|_{\p} + [u]_{\s,\p}.
$$
It is easy to see that $W^{\s,\p}(\R^n)$ is a separable and reflexive Banach space. See \cite{CERESABONDER23, CHAKER23}.

As before, given $\Omega\subset \R^n$ a bounded domain, we define $W^{\s,\p}_0(\Omega)$ as the closure of $C^\infty_c(\Omega)$ in $W^{\s,\p}(\R^n)$.

The following two theorems represent analogs to the classical Poincaré inequality and the Rellich-Kondrashov type theorem within the context of $L^{\p}(\Omega)$ and anisotropic fractional Sobolev space.

\begin{prop}[Poincaré]
	
Given $\Omega$ an open bounded subset on $\R^n$, there exists constants $C_1(\Omega,\p,n)>0$ and $C_2(\Omega,\p,\s,n)>0$ such that for every $u$ in $W^{1,\p}_0(\Omega)$, the following inequality holds
\begin{equation}\label{Poincare1}
\|u\|_{\p} \leq C_1\|\nabla u\|_{\p},
\end{equation}
and for every $u$ in $W^{\s,\p}_0(\Omega)$, the following inequality holds:
\begin{equation}\label{Poincare2}
\|u\|_{\p} \leq C_2 [u]_{\s,\p}.
\end{equation}
\end{prop}

\begin{proof}
Let $u$ be a function in $W^{1,\p}_0(\Omega)$. On one hand, observe that since $p_i\le p_n$ for every $i=1,\dots,n-1$ and $|\Omega|<\infty$, it follows by H\"older's inequality that $L^{p_n}(\Omega)$ is continuously embedded in $L^{\p}(\Omega)$, that is, there exists a positive constant $C>0$ such that
$$
\|u\|_{\p}\leq C\|u\|_{L^{p_n}} .
$$
On the other hand, the Poincaré inequality for functions in $W^{1,p_n}_0(\Omega)$
$$
\|u\|_{L^{p_n}(\Omega)}\leq C \|u_{x_n}\|_{L^{p_n}(\Omega)}\leq C \sum_{i=1}^{n}\|u_{x_i}\|_{L^{p_i}(\Omega)}
$$
Therefore, by combining these results, we obtain \eqref{Poincare1}.
	
For the second inequality, let $u$ be a function in $W^{\s,\p}_0(\Omega)$, we can assume that there exist $R>0$ such that $\supp{u}\subset Q_R = [-R,R]^n$. Hence,
\begin{align*}
[u]_{s_1,p_1,1}^{p_1}&=\int_{\R^n}\int_\R\frac{|u(x+he_1)-u(x)|^{p_1}}{|h|^{1+s_1p_1}}\,dh\,dx\\
&\geq \int_{Q_R}\int_\R\frac{|u(x+he_1)-u(x)|^{p_1}}{|h|^{1+s_1p_1}}\,dh\,dx\\
&\geq \int_{Q_R'}\int_{|x_1|\leq R}\int_{|x_1+he_1|\geq R}\frac{|u(x)|^{p_1}}{|h|^{1+s_1p_1}}\,dh\,dx_1\,dx'\\
&\geq \int_{Q_R'}\int_{|x_1|\leq R}|u(x)|^{p_1}\int_{|h|\geq 2R}\frac{1}{|h|^{1+s_1p_1}}\,dh\,dx_1\,dx'\\
&\geq C\|u\|_{p_1}^{p_1},
\end{align*}
where $Q_R' = [-R,R]^{n-1}$ and $dx' = dx_2\cdots dx_n$.

Arguing in a similar fashion we conclude that there exists $C_i(\Omega, s_i, p_i)$ such that
$$
C_i[u]_{s_i,p_i,i}\geq \|u\|_{p_i}.
$$
Therefore taking $K=\max_{i}\{C_i\}$ we have that 
$$
K\sum_{i=1}^n [u]_{s,p,i}\geq \sum_{i=1}^n \|u\|_{p_i}\geq \|u\|_{p_n}\geq  C \|u\|_{\p}.
$$

This fact concludes the proof of \eqref{Poincare2}.
\end{proof}

The following notation will be used. Given a vector $\q = (q_1,\dots,q_n)$ with $q_i>0$ for $i=1,\dots,n$, we denote by $\bar\q$ the {\em harmonic mean} of the vector $\q$, i.e. 
$$
\bar\q := \left(\frac{1}{n}\sum_{i=1}^n\frac{1}{q_i}\right)^{-1}.
$$

Next, given two vectors $\q = (q_1,\dots,q_n)$ and $\rr =  (r_1,\dots, r_n)$ with $q_i, r_i>0$ for $i=1,\dots,n$ we define the {\em product} $\q\rr$ as
$$
\q\rr = (q_1 r_1,\dots,q_n r_n),
$$
the coordinate by coordinate multiplication.

\begin{prop}[Rellich-Kondrashov]
Let $\p=(p_1,\dots,p_n)$ with $1<p_i<\infty$, $i=1,\dots,n$ and be such that
\begin{equation}\label{condicion1}
\bar\p<n.
\end{equation}
Define the {\em critical exponent} $\p^*$ as	
\begin{equation}\label{p*}
\p^*:=\frac{n\bar \p}{n-\bar \p}.
\end{equation} 	
Then $W^{1,\p}_0(\Omega)\subset L^{q}(\Omega)$, for all $1\le q\leq \p^*$. Even more $W^{1,\p}_0(\Omega) \subset\subset L^{q}(\Omega)$ if $1\le q<\p^*$. In particular $W^{1,\p}_0(\Omega) \subset\subset L^{\p}(\Omega)$.
	
Now, let $\s = (s_1,\dots,s_n)$ with $0<s_i<1$, for $i=1,\dots,n$ and $\p$ be as before. Assume that 
\begin{equation}\label{condicion2}
\overline{\s\p} < n
\end{equation} 
and define the {\em fractional critical exponent}
\begin{equation}\label{ps*}
\p^*_{\s} =\frac{n\frac{\overline{\s\p}}{\bar \s}}{n-\overline{\s\p}}.
\end{equation}
Moreover, assume that 
\begin{equation}\label{condicion3}
p_n < \p^*_{\s}.
\end{equation} 
Then $W^{\s,\p}_0(\Omega) \subset L^{q}(\Omega)$, for all $1\le q\leq \p^*_{\s}$. Even more, $W^{\s,\p}_0(\Omega) \subset\subset L^{q}(\Omega)$ for $1\le q<\p^*_{\s}$. In particular	$W^{\s,\p}_0(\Omega) \subset\subset L^{\p}(\Omega)$.
\end{prop}

\begin{proof}
The proof of $W^{1,\p}_0(\Omega) \subset\subset L^{q}$ for all $1<q< \p^*$ is studied in the previous references \cite{TROISI69,ELHAMIDI2007}. To prove that $W^{1,\p}_0(\Omega) \subset\subset L^{\p}(\Omega)$ observe that as $p_n< \p^*$ then $L^{p_n}(\Omega)\subset L^{\p}(\Omega)$ continuously. 

The proof of fractional case is immediate of \cite[Theorem 2.1]{CHAKER23} and the previous idea.
\end{proof}

Without loss of generality, we can always assume that \eqref{condicion} is satisfied.

In the rest of the paper, it will always be assumed that conditions \eqref{condicion1}, \eqref{condicion2} and \eqref{condicion3} hold.

\section{The Euler-Lagrange equation}\label{section3}

\subsection{Non-fractional case}
In this subsection we will establish the Euler-Lagrange equation associated to the Rayleigh-type quotient $\Q_{\p}$ defined in \eqref{Qp}.
In fact, following ideas from \cite{LINDQVIST13} (see also \cite{Bonder-Salort-Vivas}), we show that the EL equation turns out to be the following
\begin{equation}\label{P}
\begin{cases}
-\LL_{\p} u:=\lambda \F_{\p} (u)\quad &\text{ in }\Omega\\
u=0&\text{ in } \R^n\setminus \Omega,
\end{cases}
\end{equation}
where 
\begin{equation}\label{L_p}
-\LL_{\p} u := -\diver\left(\sum_{i=1}^{n}\left|\frac{u_{x_i}}{\| u_{x_i}\|_{p_i}}\right|^{p_i-2}\frac{u_{x_i}}{\| u_{x_i}\|_{p_i}}\right)
\end{equation}
and 
\begin{equation}\label{F_p}
\F_{\p}(u)= \prod_{i=1}^{n} I_i(u)^{p_{i+1}-p_i} |u|^{p_1-2}u,
\end{equation}
where $p_{n+1}=1$.
 
\begin{defn}
Let $u$ be a function in $W^{1,\p}_0(\Omega)$,  then $u$ is a weak solution of \eqref{P} if and only if $u$ verifies  
$$
\int_{\Omega}\sum_{i=1}^{n}\left|\frac{u_{x_i}}{\|u_{x_i}\|_{p_i}}\right|^{p_i-2}\frac{u_{x_i}}{\|u_{x_i}\|_{p_i}}v_{x_i}\,dx=	\lambda\int_{\Omega}\F_{\p}(u) v\,dx,
$$
for all $v\in W^{1,\p}_0(\Omega).$
\end{defn}

We will need the following lemma regarding the behavior of the functional $\F_{\p}$.
\begin{lema}\label{Fp.bien.def}
Let $\p=(p_1,\dots,p_n)$ be such that $1<p_i<\infty$ and let $\p'=(p_1',\dots,p_n')$. Let $\F_{\p}$ be the functional defined in \eqref{F_p}. 

Then $\F_{\p}\colon L^{\p}(\R^n)\to L^{\p'}(\R^n)$ is continuous.
\end{lema}

\begin{proof}
To see that it is well defined, just observe that if $u\in L^{\p}(\R^n)$, then
\begin{align*}
\left(\int_\R |\F_{\p}(u)|^{p_1'}\, dx_1\right)^{1/p_1'} &= \prod_{i=1}^n I_i(u)^{p_{i+1}-p_1}\left(\int_\R |u|^{(p_1-1)p_1'}\, dx_1\right)^{1/p_1'}\\
&=  \prod_{i=1}^n I_i(u)^{p_{i+1}-p_1} I_1(u)^{p_1/p_1'}\\
&=  \prod_{i=2}^n I_i(u)^{p_{i+1}-p_1} I_1(u)^{p_2-1}.
\end{align*}
Iterating this procedure, one easily conclude that
$$
\|\F_{\p}(u)\|_{\p'} = \|u\|_{\p}.
$$

In order to see the continuity of $\F_{\p}$, let $\{u_k\}_{k\in\N}\subset L^{\p}(\R^n)$ be such that $u_k\to u$ in $L^{\p}(\R^n)$. Then, define
\begin{align*}
& \tilde I_1(k) := \left(\int_\R |\F_{\p}(u_k)-\F_{\p}(u)|^{p_1'}\, dx_1\right)^{1/p_1'}\\
& \tilde I_{i+1}(k) := \left(\int_\R \tilde I_i(k)^{p_{i+1}'}\, dx_{i+1}\right)^{1/p_{i+1}'},\qquad i=1,\dots,n-1.
\end{align*}
Observe that $\|\F_{\p}(u_k) - \F_{\p}(u)\|_{\p'} = \tilde I_n(k)$, so it is enough to show that, up to a subsequence,
\begin{equation}\label{cotas.Itilde}
\begin{split}
& \tilde I_i(k) \to 0 \text{ as } k\to\infty\quad \text{a.e. } (x_{i+1},\dots,x_n), \qquad i=1,\dots,n\\
& \text{and}\\
& \text{a.e. } (x_{i+2},\dots, x_n),\ \tilde I_i(k)(x_{i+1}) \le h_i(x_{i+1}), \quad \text{with } h\in L^{p_{i+1}'}(\R).
\end{split}
\end{equation}

In fact, let us see \eqref{cotas.Itilde} for $i=1$ and the rest will follow by induction.

By Remark \ref{rem.norma parcial}, it is easy see that $\F_{\p}(u_k)\to \F_{\p}(u)$ a.e. So in order to see that $\tilde I_1(k)\to 0$ for a.e. $x'=(x_2,\dots,x_n)$ we need to find an integrable majorant for $|\F_{\p}(u_k) - \F_{\p}(u)|^{p_1'}$ for a.e. $x'\in \R^{n-1}$.

Hence,
\begin{align*}
|\F_{\p}(u_k) - \F_{\p}(u)|^{p_1'}\le C\left(\prod_{i=1}^n I_i(u_k)^{(p_{i+1}-p_i)p_1'} |u_k|^{p_1} + \prod_{i=1}^n I_i(u)^{(p_{i+1}-p_i)p_1'} |u|^{p_1} \right)
\end{align*}
As, by Remark \ref{rem.norma parcial} $I_i(u_k)(\cdot, x_{i+2},\dots,x_n)\to I_i(u)(\cdot, x_{i+2},\dots,x_n)$ in $L^{p_{i+i}}(\R)$ for a.e. $(x_{i+2},\dots,x_n)$, using \cite[Theorem 4.9]{BREZIS}, there exists $h_i=h_i(\cdot, x_{i+2},\dots,x_n)\in L^{p_{i+1}}(\R)$ such that
$$
I_i(u_k)(x_1, x_{i+2},\dots,x_n)\le h_i(x_1, x_{i+2},\dots,x_n). 
$$
Moreover, since $u_k(\cdot,x')\to u(\cdot,x')$ in $L^{p_1}(\R)$ we obtain the existence of $h_0(x)$, $h_0(\cdot,x')\in L^{p_1}(\R)$ such that
$$
|u_k(x)|\le h_0(x).
$$
Hence
$$
|\F_{\p}(u_k) - \F_{\p}(u)|^{p_1'}\le C\left(\prod_{i=1}^n h_i^{(p_{i+1}-p_i)p_1'} h_0^{p_1} + \prod_{i=1}^n I_i(u)^{(p_{i+1}-p_i)p_1'} |u|^{p_1} \right) =: \Phi(x_1,x').
$$
Since $\Phi(\cdot,x')\in L^1(\R)$ for a.e. $x'\in\R^{n-1}$ we obtain that $\tilde I_1(k)\to 0$.

The proof of \eqref{cotas.Itilde} now follows by induction and the details are left to the reader.
\end{proof}

Knowing the definition of weak solution we can state our main result of this section.
\begin{thm}\label{teo euler lagrange}
Let $u$ be a function in $W^{1,\p}_0(\Omega).$ Then $u$ is a critical point of \eqref{Qp} if and only if $u$ is a weak solution of \eqref{P}.
\end{thm}

To prove this theorem we will use the following notation
\begin{equation}\label{HI}
H(u)=\|\nabla u\|_{\p}\quad \text{and}\quad I(u) = \|u\|_{\p}
\end{equation}
and we need to establish some lemmas that will facilitate the proof.

First, we have to show that the functionals $H$  and $ I$ are Fréchet differentiable.
\begin{lema}\label{I.diferenciable}
$I\colon L^{\p}(\Omega)\to \R$ and $H\colon W^{1,\p}_0(\Omega)\to \R$ are Gateaux differentiable away from zero and its derivatives are given by
\begin{equation}\label{eq.I'}
\frac{d}{dt}I(u+tv)|_{t=0} =  \langle I'(u),v \rangle=\int_{\R^n}\prod_{i=1}^n I_i(u)^{p_{i+1}-p_i} |u|^{p_1-2}uv\,dx,
\end{equation}
where $p_{n+1}=1$ and
\begin{equation}\label{eq.H'}
\frac{d}{dt}H(u+tv)|_{t=0}=    \langle H'(u),v\rangle=\sum_{i=1}^n   \int_{\R^n} \left|\frac{u_{x_i}}{\|u_{x_i}\|_{p_i}}\right|^{p_i-2}\frac{u_{x_i}}{\|u_{x_i}\|_{p_i}}v_{x_i}\,dx.
\end{equation}
That is, $H' = \LL_\p$ and $I' = \F_\p$.
\end{lema}

\begin{proof}
To prove \eqref{eq.I'}, let $u,v \in L^{\p}(\Omega)$ and $t\in\R$. Then, recalling Remark \ref{rem.recurrencia}, we compute
\begin{align*} 
\left.\frac{d}{dt}I_1(u+tv)\right|_{t=0} = I_1(u)^{1-p_1}\int_{\R} |u|^{p_1-2}uv\, dx_1.
\end{align*}
Next,
\begin{align*}
\left.\frac{d}{dt}I_2(u+tv)\right|_{t=0}&= I_2(u)^{1-p_2}\int_{\R}I_1(u+tv)^{p_2-1}\left.\frac{d}{dt}I_1(u+tv)\right|_{t=0}\,dx_2 \\
&= \int_{\R^2} I_2(u)^{1-p_2}I_1(u)^{p_2-p_1} |u|^{p_1-2}uv\, dx_1dx_2
\end{align*}
Therefore, by induction, we arrive at 
$$
\left.\frac{d}{dt}I_n(u+tv)\right|_{t=0} = \int_{\R^n} \prod_{i=1}^n I_i(u)^{p_{i+1}-p_i} |u|^{p_1-2}uv\, dx,
$$
where $p_{n+1}=1$ and the proof of \eqref{eq.I'} follows observing that $I_n=I$.

The proof of \eqref{eq.H'} is standard and the details are left to the reader.
\end{proof}

\begin{thm}\label{Frechet.dif}
The functionals $I$ and $H$ given in \eqref{HI} are Fr\'echet differentiable.
\end{thm}

\begin{proof}
The proof follows easily from Lemma \ref{I.diferenciable} just observing that $I'=\F_\p$ and $H' = \LL_\p$ are continuous. 
In fact, the continuity of $\F_\p$ is proved in Lemma \ref{Fp.bien.def} and the continuity of $\LL_\p$ is an easy excercise.
\end{proof}

At this point we can give a rigorous proof of Theorem \ref{teo euler lagrange}.
\begin{proof}[Proof of Theorem \ref{teo euler lagrange} ]
Recall that, since
$$
\Q_{\p}(u) = \frac{H(u)}{I(u)},
$$
using Lemma \ref{I.diferenciable}, one obtain that, if $u\neq 0$,
$$
\langle \Q_{\p}'(u), v\rangle = \frac{1}{I(u)} \left(\langle H'(u), v\rangle - \Q_{\p}(u)\langle I'(u), v\rangle\right).
$$
Hence, $u\in W^{1,\p}_0(\Omega)$ is a critical point of $\Q_{\p}$ if and only if
$$
\langle H'(u), v\rangle = \Q_{\p}(u)\langle I'(u), v\rangle.
$$
But this is the same as saying that $u$ is a weak solution to \eqref{P} with $\lambda=\Q_{\p}(u)$.
\end{proof}

\subsection{The fractional case}
Now, we will analyze the fractional case. So, we consider the Rayleigh-type quotient $\mathcal{Q_{\s,\p}}$ defined in \eqref{Qsp}, and look for the Euler-Lagrange equation associated to it.

The main result of this section is to show that the E-L equation is given by
\begin{equation}\label{Ps}
\begin{cases}
-\LL_{\s,\p} u = \lambda \F_{\p}(u)\quad &\text{ in }\Omega\\
u=0&\text{ in } \R^n\setminus \Omega,
\end{cases}
\end{equation}
where $\F_{\p}$ is given by \eqref{F_p} and
\begin{align*}
\LL_{\s,\p} u = \text{p.v.} \sum_{i=1}^n\int_{\R^n}\int_\R \left|\frac{D^{s_i,i}_h u(x)}{[u]_{s_i,p_i,i}}\right|^{p_i-2}\frac{D^{s_i,i}_h u(x)}{[u]_{s_i,p_i,i}} \frac{dh}{|h|^{1+s_i}}\,dx,\\
=\lim_{\ve\to 0}\sum_{i=1}^n\int_{\R^n}\int_{|h|>\ve} \left|\frac{D^{s_i,i}_h u(x)}{[u]_{s_i,p_i,i}}\right|^{p_i-2}\frac{D^{s_i,i}_h u(x)}{[u]_{s_i,p_i,i}} \frac{dh}{|h|^{1+s_i}}\,dx
\end{align*}
and
$$
D^{s,i}_h u(x) = \frac{u(x+he_i)-u(x)}{|h|^s}.
$$

It is shown in \cite{CERESABONDER23} that the operator $\LL_{\s,\p}$ is the fractional version of $\LL_{\p}$. This operator $\LL_{\s,\p}$ has to be understood in the weak sense, i.e. given $u, v\in W^{\s,\p}_0(\Omega)$,
$$
\langle -\LL_{\s,\p} u, v\rangle = \sum_{i=1}^n\int_{\R^n}\int_\R \left|\frac{D^{s_i,i}_h u(x)}{[u]_{s_i,p_i,i}}\right|^{p_i-2}\frac{D^{s_i,i}_h u(x)}{[u]_{s_i,p_i,i}} D^{s_i,i}_h v(x) \frac{dh}{|h|}\,dx.
$$

Again, we have to give a definition of weak solution.
\begin{defn}
Let $u$ be a function in $W^{\s,\p}_0(\Omega)$,  then $u$ is a weak solution of
\eqref{Ps} if $u$ verifies  
\begin{align*}
\sum_{i=1}^n\int_{\R^n}\int_\R \left|\frac{D^{s_i,i}_h u(x)}{[u]_{s_i,p_i,i}}\right|^{p_i-2}\frac{D^{s_i,i}_h u(x)}{[u]_{s_i,p_i,i}} D^{s_i,i}_h v(x) \frac{dh}{|h|}\,dx = \lambda \int_{\R^n}\F_{\p}(u)v\,dx,
\end{align*}
for all $v\in W^{\s,\p}_0(\Omega).$
\end{defn}

Again, we introduce the notation 
\begin{equation}\label{Hs}
H_{\s}(u)=[u]_{\s,\p}
\end{equation} 
and in an analogous form as in Lemma \ref{I.diferenciable} we have the following lemma, whose proof is left to the reader.
\begin{lema}\label{lemaf}
The functional $H_{\s}\colon W^{\s,\p}_0(\Omega)\to \R$ is Gateaux differentiable away from zero and its derivative is given by
\begin{align*}
\langle H_{\s}'(u), v\rangle = \sum_{i=1}^n\int_{\R^n}\int_\R \left|\frac{D^{s_i,i}_h u(x)}{[u]_{s_i,p_i,i}}\right|^{p_i-2}\frac{D^{s_i,i}_h u(x)}{[u]_{s_i,p_i,i}} D^{s_i,i}_h v(x) \frac{dh}{|h|}\,dx.
\end{align*}
That is $H_\s'=\LL_{\s,\p}$.
\end{lema}

Finally, we can state the Euler-Lagrange Theorem for the fractional case, and its proof is analogous to the non-fractional case and therefore is ommitted. 
\begin{thm}\label{teoeulerf}
Let $u$ be a function in $W^{\s,\p}_0(\Omega).$ Then $u$ is a critical point of \eqref{Qsp} if and only if $u$ is a weak solution of the Euler Lagrange equation  \eqref{Ps}
\end{thm}

\section{General properties of eigenvalues}\label{section4}
After having derived the Euler-Lagrange equation for each case, it becomes evident that these are eigenvalue problems, for which we can explore some properties.

We say that $\lambda \in \R$ is an eigenvalue of $\LL_{\p}$ under Dirichlet boundary conditions in the domain $\Omega$, if problem \eqref{P} admits a nontrivial weak solution $u \in W^{1,\p}_0(\Omega)$. Then $u$ is called an eigenfunction of $\LL_{\p}$ corresponding to $\lambda$. We will denote $\Sigma^{\p}$ the collection of these eigenvalues.

Similarly, we say that $\lambda \in \R$ is an eigenvalue of $\LL_{\s,\p}$ under Dirichlet boundary conditions in the domain $\Omega$, if problem \eqref{Ps} admits a nontrivial weak solution $u \in W^{\s,\p}_0(\Omega)$. Then $u$ is called an eigenfunction of $\LL_{\s,\p}$ corresponding to $\lambda$. We will denote $\Sigma^{\s,\p}$ the collection of these eigenvalues.

We begin this section by collecting some simple properties for the eigensets $\Sigma^\p$ and $\Sigma^{\s,\p}$.

\begin{prop}
$\Sigma^\p, \Sigma^{\s,\p}\subset (0,\infty)$ are closed sets.
\end{prop}
\begin{proof}
As we have done throughout the article, we will only provide the proof for the non-fractional case, leaving the fractional case for the reader.

First, let $\lambda\in \Sigma^\p$ and $u\in W^{1,\p}_0(\Omega)$ be an associated eigenfunction. So, if we take the same $u$ as a test function in the weak formulation of \eqref{P}, we obtain that
$$
\|\nabla u\|_{\p} = \lambda \|u\|_{\p}.
$$
Therefore, $\lambda>0$

Next, let us see that $\Sigma^\p$ is closed. To this end, let $\{\lambda_k\}$ be a sequence of eigenvalues such that $\lambda_k\to \lambda\in \R$ as $k\to \infty$ and $\{u_k\}_{k\in\N}\subset W^{1,\p}_0(\Omega)$ be a corresponding sequence of  $L^{\p}$-normalized eigenfunctions. Observe that  
$$
\|\nabla u_k\|_{\p}= \lambda_k \|u_k\|_{\p} = \lambda_k,
$$
from where it follows that $\{u_k\}_{k\in\N}$ is a bounded sequence in $W^{1,\p}_0(\Omega)$. Hence, passing to a subsequence, we get that
$$
u_k\rightharpoonup u \text{ in }W^{1,\p}_0(\Omega)\text{ and } u_k \to u \text{ in } L^{\p}(\Omega).
$$

From Lemma \ref{Fp.bien.def}, $\F_\p$ is continuous and we get that
$$
\langle \LL_\p u_k, v\rangle=\lambda_k\langle \F_\p(u_k),v\rangle \to \lambda\langle \F_p(u),v\rangle.
$$
As each $u_k$ is an $L^{\p}$-normalized eigenfunction, then
$$
\langle \LL_\p u_k, u_k\rangle=\lambda_k\to \lambda.
$$
Now, we make use of Lemma \ref{lemma1}, that is proved in the next section, to obtain that $u_k\to u$ in $W^{1,\p}_0(\Omega)$.

Hence, we can pass to the limit $k\to\infty$ in the weak formulation
$$
\langle \LL_\p u_k,v\rangle=\lambda_k\langle \F_p(u_k),v\rangle
$$
and get that
$$
\langle \LL_\p u,v\rangle=\lambda\langle \F_\p(u),v\rangle
$$
for any $v\in W^{1,\p}_0(\Omega)$. That is, $u$ is eigenfunction associated to $\lambda$ and the proof is complete.
\end{proof}

Now, we arrive at the main point of this section, that is the asymptotic behavior of the eigenset $\Sigma^{\s,\p}$ as the fractional parameters $\s=(s_1,\dots,s_n)$ verify that $s_i\to 1$, $i=1,\dots,n$.

To this end, we will make use of the following result which is a particular case of \cite[Theorem 3.3]{CERESABONDER23}. 

\begin{prop}\cite[Theorem 3.3]{CERESABONDER23}\label{stability.semilinal}
Let $\{\s_k\}_{k\in\N}$ be a sequence of fractional parameters $\s_k\to (1,\dots,1)$ as $k\to\infty$. Let $\p=(p_1,\dots,p_n)$ be such that $1<p_i<\infty$ for each $i=1,\dots, n$ and for each $k\in \N$ let $u_k\in W^{\s_k,\p}_0(\Omega)$ be such that
\begin{equation}\label{cota_uk}
\sup_{k\in\N} \|u_k\|_{\s_k,\p} <\infty.
\end{equation}

Then, there exists a function $u\in W^{1,\p}_0(\Omega)$ and a subsequence $\{u_{k_j}\}_{j\in\N}\subset \{u_k\}_{k\in\N}$ such that
$$
u_{k_j}\to u \quad \text{in } L^\p(\Omega) \quad \text{and}\quad \|\nabla u\|_{\p}\le \liminf_{k\to\infty} \,[u_k]_{\s_k,\p}.
$$
\end{prop}

Moreover, we also need to borrow a Lemma from \cite{CERESABONDER23}.
\begin{lema}\label{lema.clave}\cite[Lemma 5.7]{CERESABONDER23}
Let $\{\s_k\}_{k\in\N}$ be a sequence of fractional parameters satisfying that $\s_k\to(1,\dots,1)$ as $k\to\infty$ and let $v_k\in W^{\s_k,\p}_0(\Omega)$. Assume that $\{v_k\}_{k\in\N}$ satisfy \eqref{cota_uk}, and let $u\in W^{1,\p}_0(\Omega)$ be fixed.

Without loss of generality, we can assume that there existes $v\in W^{1,\p}_0(\Omega)$ such that $v_k\to v$ in $L^\p(\Omega)$ as $k\to\infty$.

Then
$$
\langle \LL_\p^{\s_k} u, v_k\rangle \to \langle \LL_\p u, v \rangle\quad\text{as } k\to\infty.
$$
\end{lema}

Hence we can state and prove the following theorem.
\begin{thm}
Let $\{\s_k\}_{k\in\N}$ be a sequence of fractional parameters $\s_k\to (1,\dots,1)$ as $k\to\infty$. Let $\p=(p_1,\dots,p_n)$ be such that $1<p_i<\infty$ for each $i=1,\dots, n$. Let $\lambda_k\in \Sigma^{\s_k,\p}$ be an eigenvalue of \eqref{Ps} such that $\lambda_k\to \lambda$ as $k\to\infty$. Then $\lambda\in \Sigma^\p$. Moreover, if $u_k\in W^{\s_k,\p}_0(\Omega)$ is a normalized eigenfunction associated to $\lambda_k$ then any $L^\p(\Omega)-$accumulation point $u$ of the sequence $\{u_k\}_{k\in\N}$, satisfy that $u\in W^{1,\p}_0(\Omega)$ and is an eigenfunction associated to $\lambda$.
\end{thm}

\begin{proof}
Let $\{\s_k\}_{k\in\N}$ be sequence of fractional parameters such that $\s_k\to (1,\dots,1)$ as $k\to\infty$ and et $\{\lambda_{\s_k}\}_{k\in\N}$ be a sequence of eigenvalues that converge to $\lambda$ as $k\to \infty$. For each $\lambda_{\s_k}$ there is a eigenfunction $u_k$, that we can assume to be $L^{\p}$-normalized.
	
Note that since $\|u_k\|_\p = 1$ and $\lambda_{\s_k}$ is convergent, and therefore bounded, the sequence $\{u_k\}_{k\in\N}$ satisfy \eqref{cota_uk}. Therefore we can apply Proposition \ref{stability.semilinal} and obtain a subsequence, that we still denote $\{u_k\}_{k\in\N}$ and a function $u\in W^{1,\p}_0(\Omega)$ such that $u_k\to u$ in $L^{\p}(\Omega)$.

Now, using Lemma \ref{lema.clave}, the proof follows by using a classical monotonicity argument.
\end{proof}

\section{Existence of eigenvalues}\label{section5}
In this section, we establish the existence of eigenvalues using the Ljusternik-Schnirelman theory. This is the main part of the article.

First we recall some abstract results from critical point theory that will be essential in our proof of the existence of eigenvalues.

Let $X$ be a reflexive Banach space and maps $\phi,\psi \in C^1(X,\R)$. We will assume that $\phi, \psi$ verify the assumptions (H1)--(H4) below:
\begin{enumerate}
\item[(H1)] $ \phi,\psi \in C^1(X,\R)$, are even maps with $ \phi(0)=\psi(0)=0$ and the level set 
$$
{\mathcal M}=\{u\in X :\psi(u)=1\}
$$ 
is bounded. \label{H1}
		
\item[(H2)] $ \phi'$ is completely continuous. Moreover, for any $u\in X$ it holds that 
$$
\langle  \phi'(u),u \rangle=0 \iff  \phi(u)=0,
$$ 
where $\langle \cdot ,\cdot \rangle$ denote the duality brackets for the pair $(X,X^*).$

\item[(H3)] $\psi'$ is continuous, bounded and, as $k\to\infty$, it holds that 
$$
u_k\rightharpoonup u, \psi'(u_k)\rightharpoonup v \text{ and }\langle \psi'(u_k),u_k\rangle \to \langle v,u\rangle \Rightarrow u_k\to u\text{ in } X. 
$$

\item[(H4)] For every $u\in X\setminus \{0\}$ it holds that
$$
\langle \psi'(u),u\rangle>0,\ \lim_{t\to\infty}\psi(tu)=\infty \text{ and } \inf_{u\in M}\langle \psi'(u),u\rangle>0.
$$		
\end{enumerate}

Now, for any $n\in\N$, we define
$$
{\mathcal K}_n =\{K\subset \mathcal M\colon K \text{ is symmetric, compact, with } \phi|_K>0 \text{ and } \gamma(K)\ge n\},
$$
where $\gamma(K)$ is the Krasnoselskii genus of the set $K$.

Finally, let
$$
c_n = \begin{cases}
\sup_{K\in{\mathcal K}_n} \min_{u\in K} \phi(u) & \text{if } {\mathcal K}_n\neq\emptyset\\
0 & \text{if } {\mathcal K}_n=\emptyset
\end{cases}
$$

The following general abstract result is proved in \cite{ZEIDLER} (see also \cite[Theorem 9.27]{MOTREANU14}).
\begin{thm}\label{prop-LS}
Let $X$ be a reflexive Banach space and $\phi, \psi \in C^1(X,\R)$. Assume that $\phi, \psi$ satisfy {\em (H1)--(H4)}. Then
\begin{enumerate}
\item $c_1 <\infty$ and $c_n\to 0$ as $n\to\infty$.

\item If $c=c_n >0$, then we can find an element $u\in \mathcal M$ that is a solution of
\begin{equation}\label{eq.abstracta}
\mu \psi'(u)=\phi'(u),\quad (\mu, u)\in \R\times{\mathcal M},
\end{equation}
for an eigenvalue $\mu\neq0$ and such that $\phi(u) = c$.

\item More generally, if $c = c_n = c_{n+k} > 0$ for some $k \ge 0$, then the set of solutions $u \in{\mathcal M}$ of \eqref{eq.abstracta} such that $\phi(u)=c$ has genus $\ge  k + 1$.

\item If $c_n > 0$ for all $n \ge1$, then there is a sequence $\{(\mu_n,u_n)\}_{n\in\N}$ of solutions of \eqref{eq.abstracta} with $\phi(u_n)=c_n$, $\mu_n\neq0$ for all $n\ge 1$, and $\mu_n\to 0$ as $n\to\infty$.

\item If we further require that 
$$
\langle \phi'(u),u\rangle=0 \text{ if and only if } \phi(u)=0 \text{ if and only if } u=0,
$$
then, $c_n > 0$ for all $n \ge 1$, and there is a sequence $\{(\mu_n,u_n)\}_{n\in\N}$ of solutions of \eqref{eq.abstracta} such that $\phi(u_n) = c_n$, $\mu_n\neq 0$, $\mu_n\to 0$, and $u_n \rightharpoonup 0$ in $X$ as $n\to\infty$.
\end{enumerate}
\end{thm}

Now we will apply Theorem \ref{prop-LS} in the case where $X=W^{1,\p}_0(\Omega)$, $(\phi, \psi)=(I,H)$ and in the case where $X=W^{\s,\p}_0(\Omega)$ and $(\phi,\psi) = (I, H_{\s})$, where the operators $I, H$ and $H_{\s}$ where introduced in \eqref{HI} and \eqref{Hs}.

For enhanced readability of the work, we will demonstrate the properties of $H$ and $H_{\s}$ through a series of lemmas.
\begin{lema}\label{H'monotone}
Let $H\colon W^{1,\p}_0(\Omega)\to\R$ be defined in \eqref{HI} and $H_{\s}\colon W^{\s,\p}_0(\Omega)\to\R$ be defined in \eqref{Hs}. 

Then $H'\colon W^{1,\p}_0(\Omega)\to W^{-1,\p'}(\Omega)$ and $H_{\s}'\colon W^{\s,\p}_0(\Omega)\to W^{-\s,\p'}(\Omega)$ are bounded and monotone.
\end{lema}

\begin{proof}
Let us first demonstrate that $H'$ is bounded. In fact
\begin{align*}
|\langle H'(u),v\rangle|& =\left|\sum_{i=1}^{n}\frac{1}{\|u_{x_i}\|_{p_i}^{p_i-1}}\int_{\R} |u_{x_i}|^{p_i-2}u_{x_i}v_{x_i}\,dx\right|\\
&\leq \sum_{i=1}^{n}\frac{1}{\|u_{x_i}\|_{p_i}^{p_i-1}}\int_{\R} |u_{x_i}|^{p_i-1}|v_{x_i}|\,dx\\
&\leq\sum_{i=1}^{n}\frac{\|u_{x_i}\|_{p_i}^{p_i/p_i'}\|v_{x_i}\|_{p_i}}{\|u_{x_i}\|_{p_i}^{p_i-1}}\\
&\leq \sum_{i=1}^{n}\|v_{x_i}\|_{p_i}=\|\nabla v\|_{\p}.
\end{align*}
Therefore, $H'$ is a bounded operator. Moreover, from this result we can obtain the monotonicity of it.
$$
\langle H'(u),v\rangle + \langle H'(v),u\rangle \leq \left|\langle  H'(u),v\rangle\right| + \left|\langle H'(v),u\rangle\right|\leq \|\nabla v\|_{\p}+\|\nabla u\|_{\p}.
$$ 
Hence,
\begin{align*}
\langle H'(u)-H'(v),u-v\rangle&=\langle H'(u),u\rangle+\langle H'(v),v\rangle-\left(\langle H'(u),v\rangle+\langle H'(v),u\rangle\right)\\
&=\|\nabla u\|_{\p}+\|\nabla v\|_{\p}-\left(\langle H'(u),v\rangle+\langle H'(v),u\rangle\right)\geq 0.
\end{align*}
The proof for $H'$ is complete.
	
The proof for $H_{\s}'$ is analogous and the details are left to the reader.
\end{proof}

\begin{lema}\label{lemma1}
The operators $H$ and $H_{\s}$ given in \eqref{HI} and \eqref{Hs} respectively, verify hypothesis {\em (H3)}. 

That is $H'\colon W^{1,\p}_0(\Omega)\to W^{-1,\p'}(\Omega)$ and $H_{\s}'\colon W^{\s,\p}_0(\Omega)\to W^{-\s,\p'}(\Omega)$ are continuous and bounded operator, and moreover, as $k\to\infty$, it holds that 
\begin{align}\label{H31}
& u_k\rightharpoonup u,\ H'(u_k)\rightharpoonup v \text{ and }\langle H'(u_k),u_k\rangle \to \langle v,u\rangle \Rightarrow u_k\to u\text{ in } W^{1,\p}_0(\Omega)\\
\label{H32}
& u_k\rightharpoonup u,\ H_{\s}'(u_k)\rightharpoonup v \text{ and }\langle H_{\s}'(u_k),u_k\rangle \to \langle v,u\rangle \Rightarrow u_k\to u\text{ in } W^{\s,\p}_0(\Omega).
\end{align} 
\end{lema}

\begin{proof}
In view of Lemma \ref{H'monotone} it remains to see that $H'$ and $H_{\s}'$ are continuous and that verify \eqref{H31} and \eqref{H32} respectively.

First we claim that $H'$ is a continuous operator. In fact just observe that we can rewrite $H'$ as
$$
H'(u) = \sum_{i=1}^n \frac{J_{p_i}(u_{x_i})}{\|u_{x_i}\|_{p_i}^{p_i-1}},
$$
where $J_p(u) := |u|^{p-2}u$.
	
Since $J_p\colon L^p(\R^n)\to L^{p'}(\R^n)$ is continuous, the claim follows

To verify \eqref{H31} let  $\{u_k\}_{k\in\N}$ be a sequence in $W^{1,\p}_0(\Omega)$ such that $u_k \rightharpoonup u$ in $W^{1,\p}_0(\Omega)$, $ H'(u_k)\rightharpoonup v$ in $W^{-1,\p'}(\Omega)$ and $\langle H'(u_k),u_k\rangle\to\langle v,u\rangle$ then we need to show that  $u_k\to u $ in $W^{1,\p}_0(\Omega).$

Given $w\in W^{1,\p}_0(\Omega)$ arbitrary, by the monotonicity of $H'$ (Lemma \ref{H'monotone}) we get that
$$
0\leq \langle H'(w) - H'(u_k), w - u_k\rangle. 
$$
Taking the limit as $k\to\infty$, we arrive at	
$$
0\leq \langle H'(w),w - u\rangle- \langle v, w - u\rangle.
$$
Now we can take $w=u+tz$ with $t>0$ and we find that
$$
0\leq \langle H'(u+tz)-v,tz\rangle 
$$
Dividing by $t$ and taking limit $t\to0+$ we get that for all $z\in W^{1,\p}_0(\Omega)$,
$$
0\leq \langle H'(u)-v,z\rangle. 
$$
Therefore $H'(u)=v$. Moreover 
$$
\|\nabla u_k\|_{\p} = \langle H'(u_k),u_k\rangle \to \langle v,u\rangle = \langle H'(u),u\rangle = \|\nabla u\|_{\p}.
$$ 
	
As the space $W^{1,\p}_0(\Omega)$ is uniformly convex, weak convergence and norm convergence implies that $u_k\to u$ strongly as desired.
	
The proof for $H_{\s}'$ is analogous.
\end{proof}

The upcoming theorem is the main important result of this section, as it assures the existence of eigenvalues.

\begin{thm}\label{existence.eigenvalues}
There exist a sequence $\{u_k\}_{k\in \N}\subset W^{1,\p}_0(\Omega)$ of critical points of $\Q_{\p}$ with associated critical values $\{\lambda_k\}_{k\in \N}\subset \R$ such that $\lambda_k\to\infty$ as $k\to\infty$. Moreover, these critical values have the following variational characterization  
\begin{equation}\label{variational characterization}
\lambda_k=\inf_{K\in {\mathcal K}_k}\sup_{u\in K} H(u)
\end{equation}
where, for any $k\in\N$
$$
{\mathcal K}_k=\{K\subset M \text{ compact, symmetric with } H'(u)>0 \text{ on } K \text{ and } \gamma(K)\geq k\}
$$
$$
M=\{u\in W_0^{1,\p}(\Omega): I(u)=1\}
$$
and $\gamma$ is the \textit{Krasnoselskii} genus of $K$.

In particular, $u_k$ is a weak solution to \eqref{P} with eigenvalue $\lambda_k$.
\end{thm} 

\begin{proof}
We must confirm that the functionals $I$ and $H$ satisfy the hypotheses of Theorem \ref{prop-LS}.

Note that conditions (H1) and (H3) are direct consequences of Lemmas \ref{I.diferenciable} and \ref{lemma1}, respectively. Condition (H4) follows directly from the definition of $H'(u)$.

In order to show that (H2) holds, just observe that if $u_k\rightharpoonup u \text{ in } W^{1,\p}_0(\Omega)$, by the compactness of the immersion $W^{1,\p}_0(\Omega)\subset\subset L^{\p}(\Omega)$, it follows that $u_k\to u$ in $L^{\p}(\Omega)$ and using Lemma \ref{Fp.bien.def}, we get that $I'(u_k)\to I'(u)$ in $L^{\p'}(\Omega)\subset W^{-1,\p'}(\Omega)$.

Finally observe that $\langle I'(u),u\rangle = I(u) = \|u\|_{\p}$. Therefore each one is zero if and only if $u=0$.

We then apply the Ljusternik-Schnirelman theory, Theorem \ref{prop-LS}, to the functionals $I$ and $H$ on the level set ${\mathcal M} = \{u\in W^{1,\p}_0(\Omega)\colon H(u)=1\}$.

 By Theorem \ref{prop-LS} there exist a sequence of numbers $\{\mu_k\}_{k\in\N}\searrow0$ and functions $\{u_k\}_{k\in\N}\in W^{1,\p}_0(\Omega)$ normalized such that $H(u_k)=1$, and
\begin{equation}\label{eq.abstracta2}
\mu_k \langle H'(u_k),v\rangle = \langle I'(u_k),v\rangle \quad \forall v \in W^{1,\p}_0(\Omega)
\end{equation}
and $I(u_k)=c_k$ with
\begin{equation}\label{ck}
c_k = \sup_{K\subset {\mathcal K_k}} \min_{u\in K} I(u).
\end{equation}
Using that $\langle H'(u),u\rangle = H(u)$ and $\langle I'(u),u\rangle = I(u)$ one immediately obtain that $c_k=\mu_k$.	 

So, if we denote $\lambda_k = \mu_k^{-1}$, using \eqref{eq.abstracta2}, we have that $u_k$ is a weak solution to \eqref{P} with eigenvalue $\lambda_k$	and from \eqref{ck} one also obtain the validity of \eqref{variational characterization}.
\end{proof}

\begin{rem}
The eigenvalues obtained in Theorem \ref{existence.eigenvalues} are commonly called the Ljusternik-Schnirelman eigenvalues or simply  the \textbf{LS-eigenvalues} and are denoted by $\Sigma^\p_{LS}$.
\end{rem}

Similarly, we state a Theorem for the fractional counterpart. The proof is a slight variation of Theorem \ref{existence.eigenvalues} and is omitted.

\begin{thm}\label{existence.eigenvaluess}
There exist a sequence $\{u_k^{\s}\}_{k\in \N}\subset W^{\s,\p}_0(\Omega)$ of critical points of $\Q_{\s,\p}$ with critical values $\{\lambda_k^{\s}\}_{k\in \N}\subset \R$ such that $\lambda_k^{\s}\to\infty$ as $k\to\infty$. Moreover, this critical values have the following variational characterization  
\begin{equation}\label{variational characterizations}
\lambda_k^{\s}=\inf_{K\in {\mathcal K}_k^{\s}}\sup_{u\in K} H_{\s}(u)
\end{equation}
where, for any $k\in\N$
$$
{\mathcal K}_k^{\s}=\{K\subset M_{\s} \text{ compact, symmetric with } H_{\s}'(u)>0 \text{ on } K \text{ and } \gamma(K)\geq k\}
$$
$$
M_{\s}=\{u\in W_0^{\s,\p}(\Omega): I(u)=1\}
$$
and $\gamma$ is the \textit{Krasnoselskii} genus of $K$.
\end{thm} 

This eigenvalues related to the fractional problem will be denoted by $\Sigma_{LS}^{\s, \p}$.

\section*{Acknowledgements}
This work was partially supported by UBACYT Prog. 2018 20020170100445BA, by ANPCyT PICT 2016-1022 and by PIP No. 11220150100032CO.

J. Fern\'andez Bonder is a members of CONICET and I. Ceresa Dussel is a doctoral fellow of CONICET.
\bibliographystyle{plain}
\bibliography{References.bib}

\begin{thebibliography}{10}

\bibitem{ADAMS88}
Robert~A. Adams.
\newblock Anisotropic sobolev inequalities.
\newblock {\em Časopis pro pěstování matematiky}, 113(3):267--279, 1988.

\bibitem{BELLONI04}
Marino Belloni and Bernd Kawohl.
\newblock The pseudo-{$p$}-{L}aplace eigenvalue problem and viscosity solutions
  as {$p\to\infty$}.
\newblock {\em ESAIM Control Optim. Calc. Var.}, 10(1):28--52, 2004.

\bibitem{BENEDEK61}
A.~Benedek and R.~Panzone.
\newblock The space lp, with mixed norm.
\newblock {\em Duke Mathematical Journal}, 28(3), 1961.

\bibitem{BREZIS}
Haim Brezis.
\newblock {\em Functional analysis, {S}obolev spaces and partial differential
  equations}.
\newblock Universitext. Springer, New York, 2011.

\bibitem{CAFA07}
Luis Caffarelli and Luis Silvestre.
\newblock An extension problem related to the fractional {L}aplacian.
\newblock {\em Comm. Partial Differential Equations}, 32(7-9):1245--1260, 2007.

\bibitem{CERESABONDER23}
Ignacio Ceresa~Dussel and Juli\'{a}n Fern\'{a}ndez~Bonder.
\newblock A {B}ourgain-{B}rezis-{M}ironescu formula for anisotropic fractional
  {S}obolev spaces and applications to anisotropic fractional differential
  equations.
\newblock {\em J. Math. Anal. Appl.}, 519(2):Paper No. 126805, 25, 2023.

\bibitem{CHAKER23}
Jamil Chaker, Minhyun Kim, and Marvin Weidner.
\newblock The concentration-compactness principle for the nonlocal anisotropic
  {$p$}-{L}aplacian of mixed order.
\newblock {\em Nonlinear Anal.}, 232:Paper No. 113254, 18, 2023.

\bibitem{CHERNYSHOV2018}
K.R. Chernyshov.
\newblock The anisotropic norm of signals: Towards possible definitions.
\newblock {\em IFAC-PapersOnLine}, 51(32):169--174, 2018.
\newblock 17th IFAC Workshop on Control Applications of Optimization CAO 2018.

\bibitem{CORDOBA04}
Antonio C\'{o}rdoba and Diego C\'{o}rdoba.
\newblock A maximum principle applied to quasi-geostrophic equations.
\newblock {\em Comm. Math. Phys.}, 249(3):511--528, 2004.

\bibitem{DINEZZA12}
Eleonora Di~Nezza, Giampiero Palatucci, and Enrico Valdinoci.
\newblock Hitchhiker's guide to the fractional {S}obolev spaces.
\newblock {\em Bull. Sci. Math.}, 136(5):521--573, 2012.

\bibitem{ELHAMIDI2007}
A.~{El Hamidi} and J.M. Rakotoson.
\newblock Extremal functions for the anisotropic sobolev inequalities.
\newblock {\em Annales de l'Institut Henri Poincaré C, Analyse non linéaire},
  24(5):741--756, 2007.

\bibitem{BONDERROSSI02}
Julian Fernandez~Bonder and Julio~D. Rossi.
\newblock A nonlinear eigenvalue problem with indefinite weights related to the
  sobolev trace embedding.
\newblock {\em Publ. Mat.}, 46(1):221--235, 2002.

\bibitem{Bonder-Salort-Vivas}
Julián Fernández~Bonder, Ariel Salort, and Hernan Vivas.
\newblock Homogeneous eigenvalue problems in orlicz-sobolev spaces.
\newblock 05 2022.

\bibitem{LINDQVIST13}
Giovanni Franzina and Peter Lindqvist.
\newblock An eigenvalue problem with variable exponents.
\newblock {\em Nonlinear Anal.}, 85:1--16, 2013.

\bibitem{FRANZINA14}
Giovanni Franzina and Giampiero Palatucci.
\newblock Fractional {$p$}-eigenvalues.
\newblock {\em Riv. Math. Univ. Parma (N.S.)}, 5(2):373--386, 2014.

\bibitem{GONZALVES19}
Patr\'{\i}cia Gon\c{c}alves.
\newblock Hydrodynamics for symmetric exclusion in contact with reservoirs.
\newblock In {\em Stochastic dynamics out of equilibrium}, volume 282 of {\em
  Springer Proc. Math. Stat.}, pages 137--205. Springer, Cham, 2019.

\bibitem{LEAN}
An~Le.
\newblock Eigenvalue problems for the {$p$}-{L}aplacian.
\newblock {\em Nonlinear Anal.}, 64(5):1057--1099, 2006.

\bibitem{MOTREANU14}
Dumitru Motreanu, Viorica~Venera Motreanu, and Nikolaos Papageorgiou.
\newblock {\em Topological and variational methods with applications to
  nonlinear boundary value problems}.
\newblock Springer, New York, 2014.

\bibitem{RAKOSNIK79}
Ji\v{r}\'{\i} R\'{a}kosn\'{\i}k.
\newblock Some remarks to anisotropic {S}obolev spaces. {I}.
\newblock {\em Beitr\"{a}ge Anal.}, (13):55--68, 1979.

\bibitem{RAKOSNIK81}
Ji\v{r}\'{\i} R\'{a}kosn\'{\i}k.
\newblock Some remarks to anisotropic {S}obolev spaces. {II}.
\newblock {\em Beitr\"{a}ge Anal.}, (15):127--140 (1981), 1980.

\bibitem{RIASCO15}
A.~P. Riascos and Jos\'{e}~L. Mateos.
\newblock Fractional diffusion on circulant networks: emergence of a dynamical
  small world.
\newblock {\em J. Stat. Mech. Theory Exp.}, (7):P07015, 26, 2015.

\bibitem{VALDINOCI13}
Raffaella Servadei and Enrico Valdinoci.
\newblock Variational methods for non-local operators of elliptic type.
\newblock {\em Discrete Contin. Dyn. Syst.}, 33(5):2105--2137, 2013.

\bibitem{TROISI69}
Mario Troisi.
\newblock Teoremi di inclusione per spazi di sobolev non isotropi.
\newblock {\em Ricerche Mat.}, 18:3--24, 1969.

\bibitem{TSIOTSIOS2013}
Chourmouzios Tsiotsios and Maria Petrou.
\newblock On the choice of the parameters for anisotropic diffusion in image
  processing.
\newblock {\em Pattern Recognition}, 46(5):1369--1381, 2013.

\bibitem{ZEIDLER}
Eberhard Zeidler.
\newblock {\em Nonlinear functional analysis and its applications. {III}}.
\newblock Springer-Verlag, New York, 1985.
\newblock Variational methods and optimization, Translated from the German by
  Leo F. Boron.

\end{thebibliography}
\end{document}